\documentclass{amsart}
\usepackage{amssymb,graphicx,rotating,multirow,multicol}
\usepackage{amsmath,amsthm,amsfonts,amscd}
\usepackage[toc,page]{appendix}
\usepackage{caption,comment}
\usepackage[mathscr]{euscript}
\usepackage{epstopdf}
\usepackage{float}

\usepackage{graphicx}
\usepackage{epigraph}
\swapnumbers
\newtheorem{theorem}{Theorem}[subsection]

\newtheorem{lemma}[theorem]{Lemma}

\newtheorem{cor}[theorem]{Corollary}
\newtheorem{proposition}[theorem]{Proposition}
\theoremstyle{remark}
\newtheorem{remark}[theorem]{Remark}
\numberwithin{equation}{subsection}

\def\eg{\emph{e.g.}}

\def\la{\langle}
\def\ra{\rangle}
\let\+\sqcup
\let\dsum\+

\def\L{\Lambda}
\def\s{\sigma}
\def\LL{\Bbb L}
\def\discr{\operatorname{discr}}

\def\id{\operatorname{id}}

\def\rk{\operatorname{rk}}

{\catcode`\@=11 \gdef\mnote#1{\marginpar{\footnotesize
 \tolerance\@M\spaceskip2.6\p@ plus10\p@ minus.9\p@\rm#1}}}
\def\Dg:{\endgraf{\bf Dg:\enspace}\ignorespaces}

\let\Bbb\mathbb
\let\Cal\mathcal

\def\G{\Gamma}

\def\s{\sigma}

\newcommand{\be}{\begin{equation}}
\newcommand{\ee}{\end{equation}}
\let\ge\geqslant 
\let\le\leqslant 

\let\la\langle
\let\ra\rangle
\let\til\widetilde

\def\Z{\Bbb Z}
\def\R{\Bbb R}

\def\Q{\Bbb Q}

\def\M{\Bbb M}

\def\L{\Lambda}

\def\Rp#1{\Bbb{RP}^{#1}}
\def\Aut{\operatorname{Aut}}

\def\conj{\operatorname{conj}}

\makeatletter
\newcommand{\addresseshere}{%
  \enddoc@text\let\enddoc@text\relax
}
\makeatother

\begin{document}

\renewcommand{\arraystretch}{1.2}
\title[Pure deformation classification]
{Chirality of Real Non-singular Cubic Fourfolds and Their Pure Deformation Classification}
\author[]{ S.~Finashin, V.~Kharlamov}
\address{
Department of Mathematics, Middle East Tech. University\endgraf
 06800 Ankara Turkey}
\email{serge@metu.edu.tr}
\address{Universit\'{e} de Strasbourg et IRMA (CNRS)\endgraf 7 rue Ren\'{e}-Descartes, 67084 Strasbourg Cedex, France}
\email{kharlam@math.unistra.fr}


\date{}

\begin{abstract} In our previous works we have classified real non-singular cubic hypersurfaces in the 5-dimensional projective space up to equivalence that includes both real projective transformations and continuous variations of coefficients preserving the hypersurface non-singular. Here, we perform a finer classification giving a full answer to the chirality problem: which of real non-singular cubic hypersurfaces can not be continuously deformed to their mirror reflection.
\end{abstract}

\maketitle
\setlength\epigraphwidth{.66\textwidth}
\epigraph{... I'll tell you all my ideas about Looking-glass House. First, there's the room you can see through the glass -  that's just the same as our drawing room, only the things go the other way ...
How would you like to live in Looking-glass House, Kitty? I wonder if they'd give you milk in there? Perhaps Looking-glass milk isn't good to drink...
}
{Lewis Carroll,
\textit{Through the Looking-Glass,  and What Alice Found There.}
(cf. note 6 on page 144 in \cite{annotated})
}

\section{Introduction}

\subsection{Chirality problem}
There are two deformation equivalence relations emerging naturally in the study of real non-singular projective hypersurfaces in the framework of 16th Hilbert's problem.
One of them is  the {\it pure deformation equivalence}
that assigns hypersurfaces
to the same equivalence class
if they can be
joined by a continuous path (called a {\it real deformation}) in the space of real non-singular projective hypersurfaces of  some fixed degree.
The other one is the
{\it coarse deformation equivalence},
in which real deformations
are combined with real projective transformations.

If the dimension of the ambient projective space is even, then the group of real projective transformations is connected,
and the above equivalence relations coincide. By contrary, if
the dimension of the ambient projective space is odd,
this group has two connected components, and
some of coarse deformation classes may split into two pure deformation classes.
The hypersurfaces in such a class are
not pure deformation equivalent to their mirror images and are called {\it chiral}.
The hypersurfaces in the other classes are called {\it achiral}, since each of them is pure deformation equivalent to
its mirror image.

The first case where a discrepancy between pure and coarse deformation equivalences
shows up is that of real non-singular quartic surfaces in 3-space (achirality of all real non-singular cubic surfaces is due to F.~Klein \cite{Klein}). In this case it was studied in \cite{Kh1, Kh2},  where it was used to upgrade
the coarse deformation classification of real non-singular quartic surfaces obtained by V.~Nikulin \cite{stability}
to a pure deformation classification.

Real non-singular cubic fourfolds is a next by complexity case.
Their deformation study was launched in
\cite{deformation}, where we classified them
up to coarse deformation equivalence. Then
in \cite{chirality} we
began the study of
the chirality phenomenon and gave complete answers for
cubic fourfolds of maximal, and almost maximal, topological complexity. The approach, which we elaborated and applied in \cite{chirality}
relies on the surjectivity of the period map for cubic fourfolds established by R.~Laza \cite{Laza} and E.~Looijenga \cite{Loo}.
(Other recent applications of the period maps to real deformation problems are given,
for example, in  \cite{components}, \cite{withSaito}, \cite{ACT}, \cite{octics}, and \cite{HR}.)

Recall that according to  \cite{deformation} there exist precisely 75 coarse deformation classes of real non-singular fourfold cubic hypersurfaces $X\subset P^5$
(throughout the paper $X$
stands both for the variety itself and for its complex point set, while $X_\R=X\cap P^5_\R$ denotes the real locus).
These classes are determined by the isomorphism type of the
pairs $(\conj^* : \M(X)\to \M(X), h\in \M(X))$
where $\M(X)=H^4(X;\Z)$ is considered as a lattice, $h\in \M(X)$ is
the {\it polarization class} that is
 induced from the standard generator
of $H^4(P^5; \Z)$, and $\conj^*$ is
induced by complex conjugation $\conj:X\to X$.
This result can be simplified further and expressed in terms of a few simple numerical invariants. Namely, it is sufficient to consider the sublattice
 $\M_+^0(X)\subset \M(X)$, $\M_+^0(X)=\{ x\in \M(X) : \conj^*x=x, xh=0\}$, and to retain only the following three invariants:
the rank $\rho$ of  $\M_+^0$, the rank
$d$ of the 2-primary part $\discr_2\M_+^0$
of the discriminant $\discr\M_+^0$ (that is $d$ is the rank of the $\Z/2$-vector space $\discr_2\M_+^0/2 \discr_2\M_+^0$), and the type, even or odd, of the discriminant form on $\discr_2\M_+^0$
(see Theorem \ref{deformation-theorem} below).

Thus, to formulate
the pure deformation classification of real non-singular cubic fourfolds,
it is sufficient to list the
triples of invariants ($\rho$, $d$, parity)
which specify the
coarse deformation classes
and to indicate which of the coarse classes are chiral, and which ones are achiral.

\subsection{Main result}
\begin{theorem}\label{main}
Among the 75 coarse deformation classes precisely 18 are chiral, and, thus,
the number of pure deformation classes is 93.
The chiral classes have pairs $(\rho,d)$ satisfying
$\rho+ d\le 12$.
The only achiral classes with $\rho + d \le 12$ are three classes with $4\le \rho=d \le 6$ and one class with $(\rho,d)=(8,4)$ and even pairity.
\end{theorem}

A complete
description of the pure deformation classes is presented in Table \ref{chirality-all},
where
the coarse deformation classes are
marked by letters $c$ and $a$:  by $c$, if the class is chiral, and by $a$, if it is achiral.
We use $\rho$ and $d$ as Cartesian coordinates
and employ bold letters to indicate even parity, while keeping normal letters for odd.
For some pairs $(\rho, d)$ there exist two
coarse deformation classes, one with even discriminant form, and another with odd, and in  this case, we put the even one in brackets.

\begin{table}[h]
\caption{Pure deformation classification via chirality}\label{chirality-all}
\resizebox{\textwidth}{!}{
\begin{tabular}{cccccccccccccccccccccc}
\boxed{\bf d}&&&&&&&&&&&&&&&&&&&&&\\
11&&&&&&&&&&a&&&&&&&&&&&\\
10&&&&&&&&&a&&a({\bf a})&&&&&&&&&&\\
9&&&&&&&&a&&a&&a&&&&&&&&&\\
8&&&&&&&a({\bf a})&&a&&a({\bf a})&&a&&&&&&\\
7&&&&&&a&&a&&a&&a&&a&&&&&&&\\
6&&&&&a&&a({\bf a})&&a&&a({\bf a})&&a&&a&&&&\\
5&&&&a&&c&&a&&a&&a&&a&&a&&&\\
4&&&a&&c&&c({\bf a})&&a&&a({\bf a})&&a&&a({\bf a})&&a&&&\\
3&&c&&c&&c&&c&&a&&a&&a&&a&&a&&\\
2&c&&c({\bf c})&&c&&{\bf c}&&c&&a({\bf a})&&a&&\bf a&&a&&a({\bf a})&\\
1&&c&&c&&&&&&c&&a&&&&&&a&&a\\
0&&&{\bf c}&&&&&&&&{\bf c}&&&&&&&&{\bf a}&\\
\\
&2&3&4&5&6&7&8&9&10&11&12&13&14&15&16&17&18&19&20&21&22\ \ \boxed{\bf \rho}\\
\end{tabular}}
\end{table}

In fact,
for all but one pairs $(\rho,d)$,
the real locus of a cubic fourfold is diffeomorphic to $\Rp4 \# a(S^2\times S^2) \# b(S^1\times S^3)$, where $a=\frac12 (\rho-d), b= \frac12 (22-\rho-d)$.
The exception is $(\rho,d, \text{parity})=(12,10,\text{even})$, in which case the real locus
is diffeomorphic to $\Rp4 \sqcup S^4$
(see \cite{topology}).
Comparing this with Table \ref{chirality-all} we come to the following conclusion.

\begin{cor}
Chirality of a cubic $X\subset P^4$ is determined by the topological type of its real locus $X_\R$
unless $X_\R=\Rp4 \# 2(S^2\times S^2) \# 5(S^1\times S^3)$, or equivalently, $(\rho,d)=(8,4)$.
If $(\rho,d)=(8,4)$, then $X$ is achiral in the case of even parity, and chiral in the case of odd. \qed
\end{cor}

\subsection{Overview}
Our methods are based on the lattice technique developped in \cite{chirality} for analysis of chirality.
The power of this technique was demonstrated in \cite{chirality}  by treating the case of maximal and submaximal cubics (overall 9 coarse deformation classes and, respectively, 14 pure deformation classes). However, the methods used there were not sufficient to treat efficiently all the remaining coarse deformation classes (66 ones).
In this paper we develop further this lattice technique and implement several new elements.

First of all, it consisted in elaborating of some sufficient conditions for a sublattice or a superlattice to inherit achirality (see Section \ref{reductions-extensions}). This allowed us not only to reduce essentially the number of lattices to be analyzed, but also to replace a number of lattices of high complexity by lattices of much lower complexity.
However, two of these lattices still involved
essential complications
because of a very large, and probably even infinite, volume of the corresponding fundamental polyhedra.
To deal with these two cases, we make use in Section \ref{4-roots-added}
of
a trick allowing to reduce the size of the fundamental polyhedron by enlarging the reflection group and in the same time to keep control
on necessary information on the symmetries of the initial polyhedron (see Section \ref{4-roots-added}).

The paper is organized as follows. In Section 2 we review and develop further the
methods of \cite{chirality}.
 The proof of the main results is divided into two parts respectively to the parity of the
2-primary part of the discriminant of the lattice $\M^0_+(X)$: in Section 3 we treat the even case, and in Section 4 the odd case.
Our strategy there is, first, to use the results of  Section \ref{reductions-extensions} to reduce the study of chirality to some smaller generating set of lattices and,
second, to apply the same approach as in \cite{chirality} which is based on analysis
of symmetries of fundamental polyhedra of the group generated by reflections
determined by 2- and 6-roots of the lattice.
In Section 5 we pose some related chirality questions and partially respond to them.
In particular, in Section \ref{strong} we note that the methods developed open a way to attack
other open problems: {\it constructing maximal real projective
cubic threefolds in all dimensions, and performing deformation classification of real affine cubic threefolds
transversal to the infinity hyperplane}.

{\it Acknowledgements.} The main part of this research was completed during the first author's visit to Universit\'e de Strasbourg. The last touch was made during Research
in Pairs in Mathematisches Forschungsinstitut Oberwolfach. 
We thank these institutions for their hospitality. 

The second author was partially funded by the grant ANR-18-CE40-0009 of {\it Agence Nationale de Recherche}.
\vskip5mm

\section{Preliminaries}\label{prelim}
\subsection{Notation} This work is a continuation of \cite{chirality}. Since we make use of the results obtained there repeatedly,
we preserve the notation of that paper. In particular, as in \cite{chirality},
by $A_n, D_n, E_n$ we denote the standard classical positive definite lattices and
use the sign $+$ for orthogonal direct sum.

\subsection{Lattices under consideration}\label{lattice-notations}
Consider a non-singular cubic fourfold $X\subset P^5$.
As is known, there exists a lattice isomorphism
between $\M(X)=H^4(X)$ and $\M=3I+2U+2E_8$ which sends the
polarization class $h(X)\in \M(X)$ to $h=(1,1,1)\in 3I$. It follows then that
the primitive sublattice $\M^0(X)=\{x\in \M(X)\,|\,xh=0\}$ is
isomorphic to $\M^0=A_2+2U+2E_8$.

For $X$ defined over the reals, the complex conjugation
$\conj :X\to X$ induces a lattice involution $\conj^* :\M(X)\to \M(X)$ such that
$\conj^*(h)=h$ and, hence, induces also a lattice involution in
$\M^0(X)$. We denote by $\M_\pm^0(X)$ and $\M_\pm(X)$ the
eigen-sublattices $\{x\in \M^0(X)\, |\,
\conj^*(x)=\pm x\}$ and
$\{x\in \M(X)\, |\,
\conj^*(x)=\pm x\}$, respectively. We have
obviously $\M_-=\M_-^0$ and
$\sigma_-(\M_+(X))=\sigma_-(\M_+^0(X))$, where  $\sigma_-$ denotes
the negative index of inertia (i.e., the number of negative
squares in a diagonalization).

It was shown in \cite[Theorem 4.8]{deformation} that
two non-singular real
cubic fourfolds,
$X$ and $Y$, are
coarse deformation equivalent if and only if the lattices $\M_-(X)$ and $\M_-(Y)$ are isomorphic.
In terms of $\M^0_+$  this criterion can be translated as follows.

 \begin{theorem}\label{deformation-theorem}
The coarse deformation class of a real non-singular cubic fourfold $X$ is determined by the following three invariants:
the rank $\rho$ of the lattice $\M_+^0(X)$, the rank
$d$ of the 2-primary part $\discr_2 \M_+^0(X)$ of $\discr\M_+^0(X)$, and the type, even or odd, of  $\discr_2 \M_+^0(X)$.
\end{theorem}

\begin{proof} As it follows from Nikulin's uniqueness theorem \cite[Th. 3.6.2]{stability},
the isomorphism class of the lattice $\M_-(X)$, as of any even 2-elementary hyperbolic lattice,
is determined by its rank, the rank of its discriminant form, and the parity of the latter.
Hence, there remain to notice that $\discr_2\M_-(X)=-\discr_2\M_+^0(X)$.
\end{proof}

To determine the isomorphism class of $\M_+^0(X)$ using exclusively these three numerical invariants, we use the following uniqueness statement.

\begin{proposition}\label{determination} Let $\LL$  and $ \LL'$ be even non-degenerate lattices
whose discriminant groups split into direct sums of groups of exponents 2 and 3.
If $\sigma_-(\LL)=\sigma_-(\LL')=1$, $\rk \LL=\rk \LL'$, $\discr_3\LL=\discr_3\LL'=\discr_3 \langle 6\rangle $, $\rk \discr_2 \LL=\rk \discr_2 \LL'$, and both $\discr_2\LL, \discr_2\LL'$
are of the same parity, then $\LL$ and $\LL'$ are isomorphic.
\end{proposition}

\begin{proof}
It is trivial for lattices of rank $1$. For lattices of rank $2$, it follows from classification of binary integral quadratic forms (see, e.g., \cite[Section 15]{CS})
by passing to a reduced form
(using that $\vert \det\LL\vert \le 6$ if $\rk\discr_2\LL\le 1$ and dividing the quadratic form
of lattices by $2$ if $\rk\discr_2\LL=2$). For lattices of rank $\rk\LL\ge 3$ with $\rk\discr_2\LL < \rk \LL$, the claim of proposition follows from Nikulin's theorem \cite[Theorem 1.14.2]{stability}. In the remaining case, $\rk\discr_2\LL =\rk\LL \ge 3$ it is sufficient to divide the quadratic form of lattices by $2$ and apply the same Nikulin's theorem
if the obtained integral lattice is even, and \cite[Theorem 1.16.10]{stability} otherwise.
\end{proof}

Following \cite{chirality},  by $\Aut^+(\M^0)$ we denote the group of
those automorphisms of $\M^0$ which preserve a simultaneous orientation of negative definite planes in $\M^0$, and  put
$\Aut^-(\M^0)=\Aut(\M^0)\setminus\Aut^+(\M^0)$.

We call a lattice involution $c:\M\to \M$ {\it geometric } if
$c(h)=h$ and $\sigma_-(\M_\pm^0(c))=1$, where $\M_\pm^0(c)$
denotes the eigen-sublattices $\{x\in \M^0 \,|\,c(x)=\pm x\}$. Let
us note that all geometric involutions preserve $\M^0$ and the
involutions induced in $\M^0$ belong to $\Aut^-(\M^0)$.
As was shown
in \cite[Lemma 3.1.1 and Theorem 3.1.2]{chirality}, a
pair $(c:\M\to \M, h\in\M)$
is isomorphic to
a pair $(\conj^* :\M(X)\to \M(X), h(X)\in \M(X))$ for
some non-singular real
cubic fourfold
$X$ if and only if $c$ is a geometric involution.

 A pair $(c:\M\to \M,
h\in \M)$ isomorphic to  $(\conj^* :\M(X)\to \M(X), h(X))$ is
called the {\it homological type} of $X$.

\subsection{Lattice characterization of chirality}\label{lattice-characterization}
In what follows $\LL$ is an even lattice of signature $(n,1)$,
$n\ge1$, whose discriminant
$\discr(\LL)$ splits as $\discr_2(\LL)+\discr_3(\LL)$, where
$\discr_2(\LL)$ is a
group of period 2, and $\discr_3(\LL)=\Z/3$.

We let $\LL_\R=\LL\otimes\R$ and consider the cone
$\Upsilon=\{p\in \LL_\R\,|\,p^2<0\}$, together with the associated hyperbolic spaces
$\L=\Upsilon/\R^*$
and $\L^\#=\Upsilon/\R_+$ where
$\R^*=\R\smallsetminus\{0\}$ and $\R_+=(0,\infty)$.
In this
context, given $v\in \LL$ with $v^2>0$ we use notation $H_v$ for the hyperplane $\{p\in
\LL_\R\,|\,vp=0\}$ and $H_v^\pm$ for the half-spaces $\{p\in
\LL_\R\,|\,\pm vp\ge0\}$. For $p\in\Upsilon$, $H_v$, etc., we
use notation $[p]\in\L$, $[H_v]\subset\L$, $[p]^\#\in\L^\#$,
$[H_v]^\#\subset\L^\#$, etc., for the corresponding object after
projectivization.

We associate with each $v\in \LL$, $v^2>0$,
the reflection $R_v : \LL\otimes \Q\to \LL\otimes \Q, x\mapsto x-2\frac{xv}{v^2}v$, across the
hyperplane $H_v$.
It preserves the lattice $\LL$ invariant and belongs to the automorphism group
$\Aut(\LL)$  if
$v^2=2$, or
if $v^2=6$ and $xv$ is divisible by $3$ for all $x\in
\LL$.
We call such lattice elements
{\it 2-roots} and {\it 6-roots}, denote their sets by $V_2$ and $V_6$, and let $\Phi=V_2\cup V_6$.

Reflections $R_v$, $v\in\Phi$, generate
the {\it reflection group} $W\subset\Aut(\LL)$ which, as known (see, \eg, \cite{borel}), acts discretely in
both $\L=\Upsilon/\R^*$
and $\L^\#=\Upsilon/\R_+$.
The hyperplanes
$[H_v]$ (respectively $[H_v]^\#$), $v\in\Phi$, form a locally finite arrangement
cutting $\L$
(respectively $\L^\#$) into open polyhedra, whose closures are
called the {\it cells}. The cells in $\L$ (respectively in $\L^\#$) are the fundamental
chambers of $W$.

A cell $P\subset\L$ being fixed, the group $\Aut(\LL)$
splits into a semi-direct product $W\rtimes\Aut(P)$, where
$\Aut(P)=\{g\in\Aut(\LL)\,|\,
g(P)=P\}$ is the stabilizer of $P$.

The preimage of $P$ in $\L^\#$ is the union of a pair of
cells, $P^\#$ and $-P^\#$. Each $g\in\Aut(P)$
either permutes
$P^\#$ and $-P^\#$, and then we say that it is
{\it $P$-reversing}, or it preserves both $P^\#$ and
$-P^\#$, and then we call it {\it $P$-direct}. The
subgroup of $\Aut(P)$ formed by $P$-direct elements will
be denoted by $\Aut^+(P)$, while the coset of
$P$-reversing elements will be denoted by $\Aut^-(P)$.

An additional characteristic of
$g\in\Aut(\LL)$ is its 3-primary component, $\delta_3(g)\in\Aut(\Z/3)$,
which may be trivial or not.
We say that
$g\in\Aut(\LL)$ is {\it $\Z/3$-direct} if $\delta_3(g)=\id$, and
{\it $\Z/3$-reversing}  if $\delta_3(g)\ne\id$
(that is $\delta_3(g)=-\id$).

The following theorem is an equivalent reformulation of Theorem 4.4.1 in \cite{chirality}.

 \begin{theorem}\label{chirality-theorem}
 A non-singular real cubic fourfold
$X$
is achiral if and only if the lattice
$\M^0_+(X)$ admits
a $\Z/3$-reversing automorphism $g\in\Aut^+(P)$ for some {\rm (}or equivalently, for any{\rm )} of the cells $P$ of $\L$.
\qed
\end{theorem}

\subsection{Chirality of a lattice}
Let us pick a cell $P\subset\L$ and fix a covering cell
$P^\#\subset\L^\#$. Choosing any vector $p\in\Upsilon$ so that
$[p]^\#$ lies in the interior of $P^\#$, we let
 $\Phi^\pm=\{v\in \Phi|\pm vp>0\}$. The minimal subset $\Phi^b\subset\Phi^-$
such that $P^\#=\cap_{v\in\Phi^b}\,[H_v^-]^\#$ is called the {\it
basis of $\Phi$ defined by $P^\#$}.
The hyperplanes $[H_v]$,
$v\in\Phi^b$, support $n$-dimensional faces of $P$.
Note that any $v\in \Phi^-$ is a
linear combination of the roots in $\Phi^b$ with non-negative
coefficients.

Theorem \ref{chirality-theorem} motivates the following definitions. We call an automorphism
of $\LL$ \emph{achiral} if it is $\Z/3$-reversing and $P$-direct for some cell $P$. Respectively,
a lattice $\LL$ is called \emph{achiral} if it admits an achiral automorphism, and called {\it chiral} otherwise.

By definition, the  {\it
Coxeter graph} $\G$ of $\LL$ has $\Phi^b$ as the vertex set. The vertices
of $\G$ are colored: $2$-roots are white and $6$-roots are black. The
edges are weighted: the weight of an edge connecting vertices
$v,w\in\Phi^b$ is $m_{vw}=4\frac{(vw)^2}{v^2w^2}$, and $m_{vw}=0$
means absence of an edge. These weights are non-negative integers,
because $2\frac{ vw}{v^2},2\frac{vw}{w^2}\in\Z$, and $v^2,w^2>0$
for any $v,w\in\Phi^b$. In the case of $m_{vw}=1$, the angle
between $H_v$ and $H_w$ is $\pi/3$, and $v^2=w^2$; such edges are
not labelled.
 The case of $m_{vw}=2$ (which corresponds to angle $\pi/4$) cannot
happen, since $v^2,w^2\in\{2,6\}$. An edge of weight $m_{vw}=3$
connects always a $2$-root with a $6$-root; it corresponds to
angle $\pi/6$, and will be labelled by $6$.
 The case of $m_{vw}=4$ corresponds to parallel hyperplanes in $\L$, and we sketch a
thick edge between $v$ and $w$. If $m_{vw}>4$, then  the
corresponding hyperplanes in $\L$ are ultra-parallel (diverging),
and we sketch a dotted edge.

For a subset $J\subset \Phi^b$ we may also consider the polyhedron
$P^\#(J)=\cap_{v\in J}\,[H_v^-]^\#$
and the subgraph
$\G_J$ of $\G$ spanned by $J$.
We say that $\G_J$ is the {\it Coxeter's graph of $J$}.
A permutation $\s :J\to J$ will be called a {\it symmetry of
$\G_J$} if it preserves the weight of edges and the length of the
roots, i.e., $(\s(v))^2=v^2$ and $m_{\s(v)\s(w)}=m_{vw}$ for all
$v,w\in J$.

\begin{theorem}\label{existence of symmetry}\cite{chirality}
Let a cell $P\subset\L$ and a covering cell
$P^\#\subset\L^\#$ be fixed.
Then each $P$-direct automorphism of $\LL$ permutes the elements of $\Phi^b$
and yields a symmetry of $\G$.
Conversely, if a subset
$J\subset\Phi^b$ spans $\LL$ over $\Z$,
then any symmetry
of $\G_J$ is induced by
a unique $P$-direct automorphism of
$\LL$.
\qed
\end{theorem}

When $J$ as in Theorem \ref{existence of symmetry} spans $\LL$ over $\Z$, we call
a symmetry $\sigma :J\to J$ {\it chiral} or {\it $\Z/3$-direct} ({\it achiral} or {\it $\Z/3$-reversing})
if the $P$-direct automorphism of $\LL$ induced by $\sigma$ has the corresponding property.

To recognize $\Z/3$-reversing symmetries of $\G$, one can use the
following observation. Considering some direct sum decomposition
of $\LL$, we observe that one of the direct summands, $\LL_1$, has
$\discr_3(\LL_1)=\Z/3$, while the other direct summands have
discriminants of period 2 (because $\discr(\LL)$ gets an induced
direct sum decomposition). For any vertex $w$ of $\G$ viewed as a
vector in $\LL$, we can consider its $\LL_1$-component. This leads to the following conclusion.

\begin{proposition}\label{reversing}
Let $\LL=\LL_1+\LL_2$ with $\discr_3\LL_1=\Z/3$. Assume that $J\subset \Phi^b$ spans $\LL$ over $\Z$
and $\sigma :J\to J$ is a symmetry of $\Gamma_J$. If for some 6-root vertex $v$ of $\G_J$
the $\LL_1$-components of $v$ and
$\s(v)$ are congruent modulo $3\LL_1$, then $\sigma$ is $\Z/3$-direct. If $v-\s(v)\notin3\LL$ for some $v\in V_6$, then
$\s$ is  $\Z/3$-reversing. \qed
\end{proposition}

\subsection{Vinberg's criterion of termination}
To select a pair of cells, $P$ and $P^\#$, and to calculate simultaneously the associated system of root vectors, $\Phi^b$,
we use Vinberg's algorithm in its original form \cite{survey}: We select first a point $p\in \Upsilon$ and choose a maximal linear independent set
of root vectors $v_1,\dots, v_k$ in $p^\perp\cap \Phi$. Further vectors that form together with $v_1,\dots, v_k$ the set $\Phi^b$
are determined by some explicit inductive procedure (see, \eg,  Section 5.3 in \cite{chirality} for details).

The following Vinberg's {\it finite volume criterion}  (stated as Theorem 5.4.1 in \cite{chirality}) assures that
we found all the root vectors.

\begin{theorem}[Vinberg \cite{survey}]\label{finiteness}
A set of root vectors $J\subset\LL$
in a hyperbolic lattice of signature $(n,1)$,
obtained at some step of Vinberg's algorithm is complete (admits no continuation) if and only if
the hyperbolic volume of the
polyhedron $P^\#(J)$
is finite. \qed\end{theorem}

We will use repeatedly the following reformulation of one of Vinberg's sufficient criteria for finiteness of the volume (see \cite[Corollary 1]{termination}).

\begin{proposition}
\label{sufficient-criterion}
 Under the assumptions of Theorem  \ref{finiteness}
 the volume of
the polyhedron $P^\#(J)$
is finite if the following two conditions are satisfied.
\begin{itemize}
\item
Each connected parabolic subgraph $\G_I$ of $\G_J$, $I\subset J$, should be a connected component
of a parabolic subgraph of rank $n-1$.
\item
For
any Lann\'er's subgraph (for example, a dotted edge) $\G_S$
spanned by
 $S\subset J$
there should exist a set of vertices $T\subset J$ which includes neither vertices of $S$ nor vertices adjacent to $S$, so that the graph $\G_T$
is elliptic and
the sum of ranks of $S$ and $T$ is $n+1$.
\qed\end{itemize}
\end{proposition}

For the list of connected parabolic and Lann\'er's  graphs see Vinberg's survey \cite{survey}.

\subsection{Chirality under reductions and extensions}\label{reductions-extensions}
Let $\Gamma$
be the Coxeter graph of $\LL$ and $\Gamma_J$
be its elliptic subgraph whose vertex set, J, spans a sublattice $\LL_{J}\subset\LL$ with the discriminant of period 2 (this means that the connected components of $\Gamma_J$ have types $A_1$, $D_{2n}$ ($n\geq 2$), $E_7$, and $E_8$). Then the orthogonal complement $\LL^J\subset \LL$ of $\LL_J$ is, like $\LL$, a hyperbolic lattice with the discriminant splitting into a direct sum of the component $\discr_{3}(\LL^J)=\Z/3$ and a component $\discr_{2}(\LL^J)$ of period 2 (the latter properties follow, e.g., from \cite[Proposition 1.15.1]{stability}).

Below we put $H_J=\cap_{v\in J} H_v$.
\begin{lemma}\label{reduction} Assume that the lattice $\LL$ is achiral, and a $P$-direct $\Z/3$-reversing $f\in Aut (P)$ is induced by a symmetry of the graph $\Gamma$ which preserves $\Gamma_J$ invariant. Then $\LL^J$ is also achiral.
\end{lemma}
\begin{proof} Since $\Gamma_J$ is elliptic, the face $P^{\#}\cap H_{J}$ of $P^{\#}$ is of the same dimension as $\LL^J$.
Therefore, from $\discr_{3}(\LL^J)=\discr_{3}(\LL)=\Z/3$ it follows that
$P^{\#}\cap H_{J}$
is contained in some cell of $\LL^J$. Let us denote the latter by $P^{J\#}$. Since $f$ preserves both $P^{\#}$ and $H_J$, it should preserve also $P^{J\#}$. The restriction $f_{|L^J}$ has the same, non-trivial, 3-primary component in $\discr_{3}(\LL^J)=\discr_{3}(\LL)=\Z/3$ as $f$. Thus, $f_{|L^J}\in Aut (P^J)$
is a $P^J$-direct $\Z/3$-reversing authomorphism, and hence $\LL^J$ is also achiral.
\end{proof}

\begin{lemma}\label{extension}
Let  $\LL_h$ be an achiral lattice,  and  let $\LL_e$ be an elliptic lattice which is generated by $2$-roots and has discriminant of period 2.
Then their direct sum, $\LL= \LL_e +\LL_h$, is also an achiral lattice.
\end{lemma}
\begin{proof} Due to assumptions made, we may start Vinberg's root sequence $\Phi^b$ from a root basis of $\LL_e$. Thus, we can identify $\LL_e$
with $\LL_J$ for a corresponding vertex set $J$
of Coxeter's graph of $\LL$,
and $\LL_h$ with the orthogonal complement $\LL^J$ of $\LL_J$.

Consider, now, a cell $P^J\subset\Lambda(\LL^J)$ and the components of
its preimage, $\pm P^{J\#}\subset
\Lambda^\#(\LL^{J})$.
By the same reason as
in Lemma~\ref{reduction}, $P^J$ contains the face $P\cap H_{J}$ of some cell $P$ in $\Lambda(\LL)$. However, here, the relation is stronger: $P\cap H_{J}=P^J$. Indeed, if a wall $H_v$ of $P$ is not disjoint from $P\cap H_{J}$, then either $v\in \LL^J$ (and so $H_v\bot H_J$), or $v\in \LL_J$ (and so $H_v\supset H_J$). To see this, let us consider a root $v=v_{J}+v^{J}\in \LL_{J}+\LL^{J}$ whose components $v_{J}+v^{J}$ are both non-zero. Since $H_v$ intersects $\Lambda(\LL^J)$, we have $(v^{J})^2>0$. And since $\LL_J$ is positive definite, we have $(v_{J})^2>0$. The both squares are even. Hence, $v^{2}>2$. If $v^{2}=6$, then, in addition, $(v^{J})^2$ and $(v_{J})^2$ are divisible by $3$, which also leads to a contradiction.

Choose a $P^J$-direct $\Z/3$-reversing automorphism $f^{J}\in Aut(P^J)$ and consider its extension $f:\LL\rightarrow \LL$
that acts as the identity map
on $\LL_J$. As it follows from above, it preserves all the walls of $P$ which are not disjoint from $P\cap H_J$. Thus, it preserves $P$. Clearly, it is $P$-direct and $\Z/3$-reversing (it has the same 3-primary component as $f^J$).
\end{proof}

\section{Cases of Even Parity}

\subsection{Lattices with even discriminant forms}
For the list of coarse deformation classes in terms of numerical invariants of $\M_-$,
we address the reader to
Fig. 1 in \cite{topology}. Since 2-primary parts of the discriminant forms of $\M_-$ are $\M^0_+$ are just opposite, we keep
from this list only the classes of even parity (type I in terminology of
\cite{topology}). Then, using the relation $\rk \M^0_+=22 -\rk \M_-$ we
apply Proposition \ref{determination}  and obtain
Table \ref{even-lattices} which shows for each of the classes of even parity
 the corresponding lattice $\M_+^0$. As in Table \ref{chirality-all} the answers are arranged in columns and rows according to the values of  $d$ and $\rho$.

\begin{table}[h]
\caption{Lattices with even discriminant form}\label{even-lattices}
\resizebox{\textwidth}{!}{
\begin{tabular}{cccccccc}
\boxed{\bf{d}}\\
10&&        &              &${U(2)+A_2+E_8(2)=U(2)+E_6(2)+D_4}$&&&\\
8&&      &${U(2)+E_6(2)} $&${U+A_2+E_8(2)=U+E_6(2)+D_4}$&                &&\\
6&&        &${U+E_6(2)}$    &${U(2)+A_2+2D_4}$  &                &&\\
4&&       &${U(2)+A_2+D_4}$&${U+A_2+2D_4}$     &${U(2)+A_2+D_4+E_8}$&&\\
2&&${U(2)+A_2}$&${U+A_2+D_4}$   &${U(2)+A_2+E_8}$   &${U+A_2+D_4+E_8}$   &${U(2)+A_2+2E_8}$&\\
0&&${U+A_2}$&               &${U+A_2+E_8}$      &                &${U+A_2+2E_8}$&\\
 \hline
 &&4        &8              &12                 &16              &20&\boxed{\bf{\rho}}\\
\end{tabular}
}
\end{table}

\begin{theorem}\label{main-even}
Table \ref{even-lattices} contains 4 chiral lattices: ${U+A_2}$, ${U(2)+A_2}$, ${U+A_2+D_4}$, and ${U+A_2+E_8}$.
All the
others are achiral.
\end{theorem}

The three lattices in the bottom row were already treated in \cite{chirality}:
$U+A_2$ and $U+A_2+E_8$ were shown to be chiral, while $U+A_2+2E_8$ to be achiral.
So, our task is to
examine
the remaining thirteen cases.

We start with one trivial example.

\begin{lemma}\label{no-roots}
The lattice $\LL= U(2)+E_6(2)$ is achiral.
\end{lemma}

\begin{proof}
This lattice has neither 2-roots nor 6-roots. Hence, here $P=\L_+(c$), and therefore $g : U(2)+E_6(2) \to U(2)+E_6(2)$ that is equal to $-\id$ on $U(2)$ and $\id$ on $E_6(2)$
is $\Z/3$-direct and it belongs to $\Aut^-(P)$. Thus, $\LL$ is achiral.
\end{proof}

\subsection{A few more direct calculations}
Here we treat the cases $U(2)+A_2$, $U+A_2+D_4$, $U(2)+A_2+E_8$, and $U(2)+A_2+D_4$
using the same approach (based on Vinberg's algorithm) as we applied in \cite{chirality} for M- and (M-1)-lattices.

\begin{lemma}\label{4lattices}
\begin{enumerate}
\item
The lattices
$U(2)+A_2$ and $U+A_2+D_4$ are chiral.
\item
The lattices $U(2)+A_2+E_8$ and $U(2)+A_2+D_4$ are achiral.
\end{enumerate}
\end{lemma}

\begin{proof}
As in \cite{chirality}, we fix standard bases: $u_1, u_2$ for $U$ and $U(2)$; $a_1, a_2$ for $A_2$;  $d_1,d_2,d_3,d_4$ for $D_4$ along with its dual $d^*_1,d^*_2,d^*_3,d^*_4$ for $D^*_4\subset D_4\otimes\Q$; and
 $e_1,\dots, e_8$ for $E_8$ along with its dual $e^*_1,\dots, e^*_8$ for $E^*_8=E_8$.
 We denote by $d_1$ the ``central'' root vector of $D_4$ and then
$d_1^*=2d_1+d_2+d_3+d_4$, while $d_i^*=d_1+\frac{d_2+d_3+d_4+d_i}2$ for $i=2,3,4$. For the expressions of $e^*_i$, see for instance
\cite{bourbaki} (or \cite[Fig.1]{chirality}).

 Each time to lunch Vinberg's algorithm, we select as the initial vertex of the fundamental domain the point
$p=u_1-u_2$.  The resulting Vinberg's sequences and their Coxeter graphs are shown below in Fig. \ref{1-8(1)}, \ref{3-6}, \ref{5-4(1)}, and \ref{2-5(1)}, respectively.
As in \cite{chirality}, we omit at level 0 of Vinberg's sequence the standard simple
root vectors: $e_i$ of $E_8$, $e_i'$ of the second copy of $E_8$, $i=1,\dots,8$,
 and $d_i$, $i=1,\dots4$, of $D_4$ (when such summands appear).

\begin{figure}[h!]
\caption{Vinberg's vectors and their Coxeter's graph for $U(2)+A_2$ }\label{1-8(1)}
\hbox{\vtop{\hsize5cm
\hbox{\includegraphics[width=0.55\textwidth]{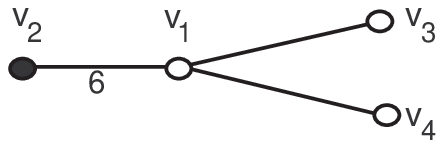}}
\hbox{
\begin{tabular}{|c|}
\hline
Parabolic of rank 2\\
\hline
$\til G_2$\\
\hline
\end{tabular}
}}
\vtop{\hsize5cm
$ \boxed{\begin{matrix}
&U(2)&A_2\\
\hline
p&1,-1&0,0\\
\text{level 0}&&\\
v_1&0,0&0,1\\
v_2&0,0&1,-1\\
\text{level 4}&&\\
v_3&0,-1&-1,-1\\
v_4&1,0&-1,-1\\
\end{matrix}}
$}}
 \end{figure}

\begin{figure}
\caption{Vinberg's vectors and their Coxeter's graph for $U+A_2+D_4$}\label{3-6}
\hbox{\vtop{\hsize4cm
\hbox{\includegraphics[width=0.7\textwidth]{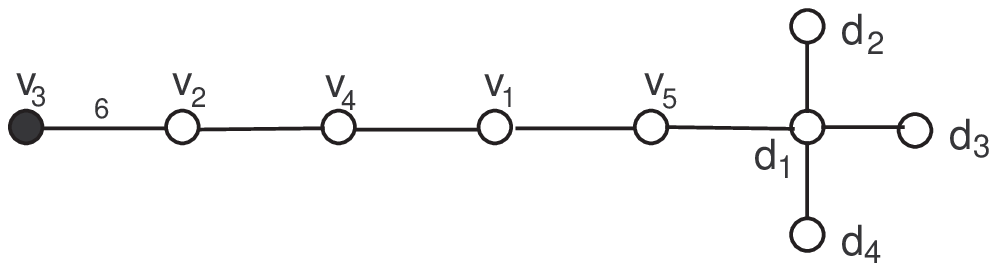}}
\hbox{
\begin{tabular}{|c|}
\hline
Parabolic of rank 6\\
\hline
$\til G_2+\til D_4$\\
\hline
\end{tabular}}}
\hskip-2cm
\vtop{\hsize4cm
$ \boxed{\begin{matrix}
&U&A_2&D_4\\
\hline
p&1,-1&0,0&0\\
\text{level 0}&&&&\\
v_1&1,1&0,0&0\\
v_2&0,0&0,1&0\\
v_3&0,0&1,-1&0\\
\text{level 1}&&\\
v_4&0,-1&-1,-1&0\\
v_5&0,-1&0,0&-d_1^*\\
\end{matrix}}
$}}
\end{figure}

\begin{figure}[h]
\caption{Vinberg's vectors and their Coxeter's subgraph for $U(2)+A_2+E_8$}\label{5-4(1)}
\hbox{
\vtop{\hsize5cm
\hbox{\includegraphics[width=0.7\textwidth]{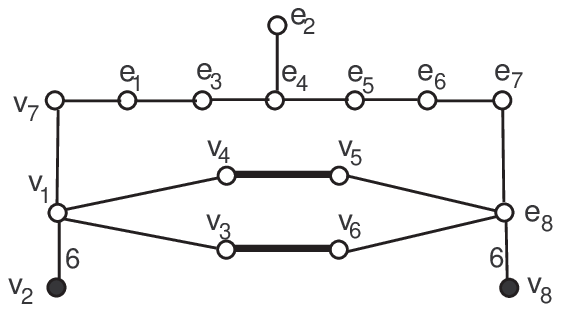}}
\hbox{
\begin{tabular}{|c|}
\hline
Parabolic of rank 10\\
\hline
$\til G_2+\til D_8$, $2\til G_2+\til E_6$, $2\til A_1+\til D_8$\\
\hline\end{tabular}}}
\hskip-2cm
\vtop{\hsize5cm
$\boxed{\begin{matrix}
&U(2)&A_2&E_8\\
\hline
p&1,-1&0,0&0\\
\text{level 0}&&\\
v_1&0,0&0,1&0\\
v_2&0,0&1,-1&0\\
\text{level 4}&&\\
v_3&0,-1&-1,-1&0\\
v_4&1,0&-1,-1&0\\
v_5&0,-1&0,0&-e_8^*\\
v_6&1,0&0,0&-e_8^*\\
\text{level 16}&&\\
v_7&1,-1&-1,-1&-e_1^*\\
\text{level 48}&&\\
v_8&3,-3&-4,-2&-3e_8^*\\
\end{matrix}}
$}}
\end{figure}

\begin{figure}[h!]
\caption{Vinberg's vectors and their Coxeter's subgraph for $U(2)+A_2+D_4$}\label{2-5(1)}
\hbox{
\vtop{\hsize 4cm
\hbox{\includegraphics[width=0.6\textwidth]{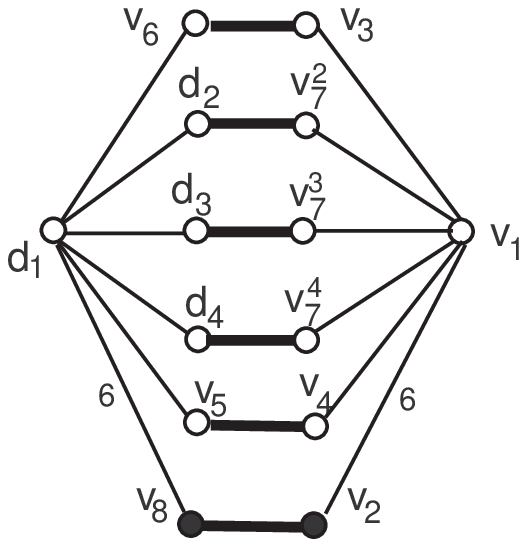}}
\hbox{
\begin{tabular}{|c|}
\hline
Parabolic of rank 6\\
\hline
$6\til A_1$, $\til G_2+\til D_4$\\
\hline\end{tabular}}}\hskip1cm
\hskip-2cm
\vtop{\hsize5cm
$\boxed{\begin{matrix}
&U(2)&A_2&D_4\\
\hline
p&1,-1&0,0&0\\
\text{level 0}&&\\
v_1&0,0&0,1&0\\
v_2&0,0&1,-1&0\\
\text{level 4}&&\\
v_3&0,-1&-1,-1&0\\
v_4&1,0&-1,-1&0\\
v_5&0,-1&0,0&-d_1^*\\
v_6&1,0&0,0&-d_1^*\\
\text{level 16}&&\\
v_7^2&1,-1&-1,-1&-2d_2^*\\
v_7^3&1,-1&-1,-1&-2d_3^*\\
v_7^4&1,-1&-1,-1&-2d_4^*\\
\text{level 48}&&\\
v_8&3,-3&-4,-2&-3d_1^*\\
\end{matrix}}
$}}
\end{figure}

(1) The completeness of Vinberg's sequences for $U(2)+A_2$ and $U+A_2+D_4$ follows from
Theorem \ref{finiteness} and Proposition \ref{sufficient-criterion}. Indeed,
in both cases the Coxeter graph does not contain Lann\'er's diagrams, and the only connected parabolic subgraphs are: $\widetilde G_2$ for $U(2)+A_2$; $\widetilde G_2$ and $\widetilde D_4$ for $U+A_2+D_4$.

Thus, by Theorem \ref{existence of symmetry} in the case of $\LL=U(2)+A_2$ the only $ P$-direct automotphism of $\LL$ is given by $v_1\to v_1, v_2\to v_2, v_3\to v_4, v_4\to v_3$.
Since due to Proposition \ref{reversing} this automorphism is $\Z/3$-direct (preserves $v_2$), we conclude that $U(2)+A_2$ is chiral.

Similarly, in the case of
$\LL= U+A_2+D_4$ each $P$-direct automorphism of $\LL$ keeps fixed the vector $v_3$. Hence, all $P$-direct symmetries are
$\Z/3$-direct, and we conclude that $U+A_2+D_4$ is chiral.

(2) When we treat  $U(2)+A_2+E_8$ and $U(2)+A_2+D_4$, we perform Vinberg's
algorithm only up to a step that provides a Coxeter subgraph which has a symmetry inducing
a $ P$-direct $\Z/3$-reversing involution, see Fig. \ref{5-4(1)} and \ref{2-5(1)}.

For $\LL= U(2)+A_2+E_8$, this is the involution
that permutes the vectors of Vinberg's sequence in accordance with the reflection in the vertical middle axis of the graph;
by Theorem \ref{existence of symmetry} it defines a $P$-direct automorphism of $\LL$  which maps $v_2$ to $v_8=-v_2\mod 3 \LL$ and by Proposition \ref{reversing} is $\Z/3$-reversing.

For $\LL=U(2)+A_2+D_4$, this is again the involution permuting the vectors of Vinberg's system in accordance with the reflection in the vertical middle axis of the graph, and it also maps $v_2$ to $v_8=-v_2\mod 3 \LL$,
and by a similar reason it is $P$-direct and $\Z/3$-reversing.
\end{proof}

\begin{remark} Coxeter's graphs on Fig. \ref{5-4(1)} and \ref{2-5(1)} are indeed complete.
\end{remark}

\begin{figure}[h!]
\caption{Vinberg's vectors and their Coxeter's subgraph for $U+A_2+2E_8$}\label{hexagon-new}
\includegraphics[width=0.65\textwidth]{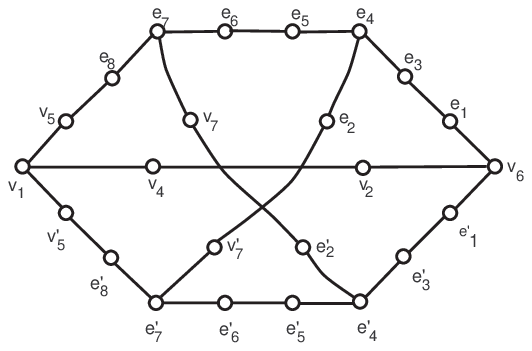}
\hskip-25mm
$\boxed{\begin{matrix}
&U&A_2&E_8&E_8\\
\hline
p&1,-1&0,0&0&0\\
\text{level 0}&&&&\\
v_1&1,1&0,0&0&0\\
v_2&0,0&0,1&0&0\\
v_3&0,0&1,-1&0&0\\
\text{level 1}&&&&\\
v_4&0,-1&-1,-1&0&0\\
v_5&0,-1&0,0&-e_8^*&0\\
v_5'&0,-1&0,0&0&-(e_8')^*\\
\text{level 16}&&&&\\
v_6&2,-2&-1,-1&-e_1^*&-(e_1')^*\\
\text{level 36}&&&&\\
 v_7&3,-3&-2,-1&-e_7^*&-(e_2')^*\\
 v_7'&3,-3&-2,-1&-e_2^*&-(e_7')^*\\
\text{level 48}&&&&\\
v_8&6,-6&-4,-2&-3e_8^*&-3(e_1')^*\\
v_8'&6,-6&-4,-2&-3e_1^*&-3(e_8')^*\\
\end{matrix}}
$
 \end{figure}

\newpage
\subsection{Achirality of lattices via
extension and reduction}\label{ext-red}

\begin{proposition}\label{reduction&extension}
Lattices $U(2)+A_2+2D_4$, $U(2)+A_2+D_4+E_8$, $U(2)+A_2+E_8(2)$, and  $U(2)+A_2+E_8$ are achiral.
\end{proposition}

\begin{proof}
To prove achirality of
lattices $U(2)+A_2+2D_4$, $U(2)+A_2+D_4+E_8$, and $U(2)+A_2+E_8(2)=U(2)+E_6(2)+D_4$ we apply Lemma \ref{extension} to $U(2)+A_2+D_4$, $U(2)+A_2+E_8$, and $U(2)+E_6(2)$ respectively. The same lemma applied to  $U(2)+A_2+E_8$ shows achirality of $U(2)+A_2+2E_8$.
\end{proof}

Figure \ref{hexagon-new} shows a hexagonal subgraph $\Gamma_{\rm hex}$ of Coxeter's graph of the lattice $U+A_2+2E_8$ which we found in
 \cite{chirality}, where we proved also the achirality of the involution
 $\Psi : \Gamma_{\rm hex} \to \Gamma_{\rm hex}$ that interchanges vertices $v_1$ and $e_4$ and keeps fixed the other
trivalent vertices (at the corners of the hexagon).

\begin{lemma}\label{hexagonal-subgraphs}
For each $n=1,2,3$, the vertex-set of
$\Gamma_{\rm hex}$
admits a subset $J$ with the following properties:
\begin{enumerate}
\item
$J$ is invariant with respect to
the achiral involution
$\Psi :\Gamma_{\rm hex} \to \Gamma_{\rm hex}$;
\item
the sublattice $\LL_J$ spanned by $J$ in $\LL$ is primitive and isomorphic to $nD_4$;
\item
 the orthogonal complement $\LL^J$ of $\LL_J$ in $\LL$
 is isomorphic to
$U+A_2+D_4+E_8$, $U+A_2+2D_4$, and $U+E_6(2)$
for $n=1, 2$, and $3$, respectively.
\end{enumerate}
\end{lemma}

\begin{proof}
The set of vertices $J=\{v_7,e_6,e_7,e_8\}\cup\{v_7',e_6',e_7',e_8'\}\cup\{e_1,e_1',v_2,v_6\}$ of
the graph $\G_{\rm hex}$ spans a sublattice
$\LL_J=3D_4$, where each $D_4$-component is invariant under the involution $\Psi$.
The embedding $\LL_J\subset \LL$
is primitive because for any subset
$S$ of $J$ there exists a vertex of $\Gamma_{\rm hex}$
adjacent only to one element
of $S$.
This proves (1) and (2) in the case $n=3$. For the cases $n=1$ and $n=2$, we take subsets of
$J$ that span just one and, respectively, two of the above
three summands $D_4$ and
get  (1) and (2)
in the same way.

Part (3) follows from  Proposition \ref{determination}.
Indeed, $\LL^J$ is even and hyperbolic,
$\discr_3 \LL^J=\discr \langle 6\rangle=\discr_3 E_6(2)$, $\discr_2 \LL^J= -\discr_2 \LL_J= \oplus_n\discr_2 D_4$,
and
$\discr_2 E_6(2)=\oplus_3\discr_2 D_4$.
\end{proof}

\begin{proposition}\label{reduction&extension2}
Lattices
$U+A_2+D_4+E_8$, $U+A_2+2D_4$, and $U+E_6(2)$ are achiral.
\end{proposition}

\begin{proof}
We deduce achirality of  $U+A_2+D_4+E_8$, $U+A_2+2D_4$, and $U+E_6(2)$ from achirality of $U+A_2+2E_8$ by
 applying Lemma \ref{reduction}.
\end{proof}

\begin{proof}[Proof of Theorem \ref{main-even}]
The lattice $U+A_2+E_8(2)=U+E_6(2)+D_4$ is obtained by adding $D_4$ to  $U+E_6(2)$.
Hence, by Lemma \ref{extension},  $U+A_2+E_8(2)$ is also achiral. All the other lattices from Table \ref{even-lattices}
are treated in Lemmas
\ref{no-roots} and \ref{4lattices}, Proposition  \ref{reduction&extension}, and in \cite{chirality}.
\end{proof}


\section{
Cases of Odd Parity and Proof of Theorem \ref{main}}

\subsection{Lattices with odd discriminant forms}

Similar to the even case, we select from the list given in \cite{topology} the classes of odd parity,
translate them in terms of the invariants ($d$, $\rho$, parity), and apply Proposition \ref{determination}
to obtain  lattices shown in Table \ref{odd-lattices}.

\begin{table}[h]
\caption{Lattices $\M_+^0$ with odd discriminant form}\label{odd-lattices}
$\boxed{\begin{matrix}
\rho-d&d\\
\hline
0&t+2& -A_1&+\la6\ra&+t A_1,\quad 0\le t\le9\\
2&t+1&-A_1&+A_2&+t A_1,\quad 0\le t\le9\\
4&t& U&+A_2&+t A_1,\quad 1\le t\le9\\
6&t+2& U&+A_2+D_4&+t A_1,\quad 1\le t\le6\\
8&t+2& -A_1&+\la6\ra+E_8&+t A_1,\quad 0\le t\le5\\
10&t+1& -A_1&+A_2+E_8&+t A_1,\quad 0\le t\le5\\
12&t& U&+A_2+E_8&+t A_1,\quad 1\le t\le5\\
14&t+2& U&+A_2+D_4+E_8&+t A_1,\quad 1\le t\le2\\
16&t+2& -A_1&+\la6\ra+2E_8&+t A_1,\quad 0\le t\le1\\
18&t+2& -A_1&+A_2+2E_8&+t A_1,\quad 0\le t\le1\\
20&t& U&+A_2+2E_8&+t A_1,\quad  t=1\\
\end{matrix}}$
\end{table}

\subsection{Achirality of lattices via extension and reduction}
Achirality of the lattices $U+A_2+E_8+A_1$ and $U+A_2+2E_8$ is already established
in \cite{chirality}.

\begin{lemma}\label{A1-stabilization}
If a lattice $\LL$ in  Table \ref{odd-lattices} with the invariants $(\rho,d)$
is achiral, then the lattice with the invariants $(\rho+1,d+1)$ (if belongs to this table) is also achiral.
\end{lemma}

\begin{proof}
The lattice with the invariants $(\rho+1,d+1)$ must be in the same row of Table  \ref{odd-lattices}
as $\LL$ but
with the value of $t$ increased by one,
that is $\LL+A_1$, and it is left to apply Lemma \ref{extension}.
\end{proof}

By Lemma \ref{A1-stabilization}, it is left
to determine in each row of Table \ref{odd-lattices}
an achiral lattice with a certain minimal value $t=t_0$, it follows then that
all the lattices with $t>t_0$  are also achiral.

\begin{figure}[h!]
\centering
\caption{Vinberg's vectors and their partial Coxeter's subgraph for $U+A_2+E_8+A_1$ }\label{6-4}
\hskip25mm\includegraphics[width=0.8\textwidth]{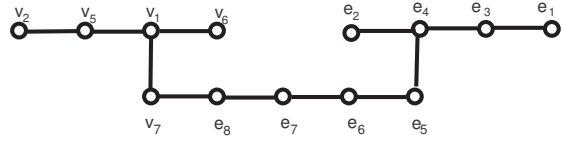}
$ \boxed{\begin{matrix}
&U&A_2&A_1&E_8\\
\hline
p&1,-1&0,0&0&0\\
\text{level 0}&&&&\\
v_1&1,1&0,0&0&0\\
v_2&0,0&0,1&0&0\\
v_3&0,0&1,-1&0&0\\
v_4&0,0&0,0&1&0\\
\text{level 1}&&&&\\
v_5&0,-1&-1,-1&0&0\\
v_6&0,-1&0,0&-1&0\\
v_7&0,-1&0,0&0&-e_8^*\\
\text{level 48}&&&&\\
v_8&6,-6&-4,-2&-3&-3(e_1)^*.
\end{matrix}}$
\end{figure}

Our main result in this section can be stated as follows.

\begin{theorem}\label{achirality-odd-complete}
In Table \ref{odd-lattices} the first achiral lattice in the first row $\rho=d$
is defined by $t=2$,  in the rows $\rho=d+2k$, $k=1,\dots,6$,
it is defined by  $\rho+d=14$, and in the last four rows all lattices are achiral.
\end{theorem}

We complete the proof of Theorem \ref{achirality-odd-complete}
in Section \ref{proof-odd-complete} after an analysis of several particular cases.

\begin{lemma}\label{reduction-1}
For any  $n=1,\dots,5$,
Coxeter's graph $\Gamma$ of the lattice $\LL=U+A_2+E_8+A_1$
contains a set $J$ of $n$  pairwise non-incident vertices so that:
\begin{enumerate}
\item
$J$ is invariant under an achiral involution  of $\LL$;
\item
$J$ generates a sublattice $\LL_J=nA_1$ primitively embedded into $\LL$;
\item
the orthogonal complement $\LL^J$ of $\LL_J$ has odd 2-discriminant form $\discr_2 \LL^J$ whose rank is
bigger by $n$ than that of
$\discr_2\LL$.
\end{enumerate}
\end{lemma}

\begin{proof}
In \cite{chirality}, Sect. 7.5, we determined the initial Vinberg's vectors of $\LL$. These vectors and Coxeter's graph
corresponding to the set of vectors $\{v_1,v_2,v_5,v_6,v_7,e_1$, $\dots,e_8\}$ are reproduced
on Fig. \ref{6-4}. This graph has an obvious reflection symmetry which induces on $\LL$ an achiral involution as it
 was shown in \cite{chirality}.

As $J$ we choose $\{e_7\}$, $\{e_6,e_8\}$, $\{v_7,e_7,e_5\}$, $\{v_1,e_6,e_8,e_4\}$, and $\{v_5, v_7,e_7,e_5, e_3\}$
for  $n=1, \dots ,5$ respectively.
It is invariant under this achiral involution.

Property (2) is satisfied for the same reasons as in the proof of Proposition \ref{hexagonal-subgraphs}.

To prove (3) we
represent $(\discr_2\LL, q= \langle \frac12\rangle$) as a result of gluing of $(\discr_2 \LL_J, q_1)$ with $(\discr_2 \LL^J, q_2)$
(see \cite[Proposition 1.3.1]{stability})
and notice that, since for any non-empty subset of $J$
the sum of its elements has odd intersection with some element of $\LL$, this gluing corresponds to a group-embedding $\discr_2\LL_J\to \discr_2\LL^J$
reversing the discriminant forms, so that the only nontrivial element
in $\discr_2\LL$ is an image of an element $e\in\discr_2\LL^J$  with $q_2(e)=q(e)=\frac12$.
\end{proof}

\begin{lemma}\label{reduction-2}
For any  $n=1, 2, 3$, Coxeter's graph $\Gamma$ of the lattice $\LL=U+A_2+2E_8$
contains a set $J$ of $n$  pairwise non-incident vertices so that:
\begin{enumerate}
\item
$J$ is invariant under an achiral involution of $\LL$;
\item
$J$ generates a sublattice $\LL_J=nA_1$ primitively embedded into $\LL$;
\item
the orthogonal complement $\LL^J$ of $\LL_J$ has odd 2-discriminant form $\discr_2 \LL^J$ whose rank equals $n$.
\end{enumerate}
\end{lemma}

\begin{proof}
For the cases $n=1,2,3$, we choose as $J$
the sets
$\{e_7\}$, $\{e_7,e_7'\}$ and $\{e_7,e_7',v_6\}$ respectively (see Fig. \ref{hexagon-new}), which are obviously invariant
with respect to the achiral involution $\Psi$ introduced in section \,\ref{ext-red}.
The proof of (2) is the same as in Lemma \ref{reduction-1},
while for (3)
it is much easier, because $\discr_2 S=-\discr_2S^\perp$ for primitive sublattices in a lattice with trivial $\discr_2$.
\end{proof}

\begin{proposition}\label{achirality-odd}
All lattices in Table  \ref{odd-lattices}  for which $\rho+d\ge14$ and $\rho-d\ge 2$ are achiral.
\end{proposition}

\begin{proof} The 6 lattices from Table  \ref{odd-lattices} with $\rho+d\ge14$ and $\rho-d\ge 2$ are the following:
$\LL= U+A_2+E_8+A_1$ and its 5 sublattices $L^J$ given by Lemma \ref{reduction-1} for $n=1,\dots, 5$. By Lemma \ref{reduction} achirality of $\LL$ (established in
\cite{chirality}) implies achirality of these $L^J$. Similarly, achirality $\LL= U + A_2+2E_8$ (also established in  \cite{chirality}) and Lemma \ref{reduction-2} imply achirality of lattices with $\rho+d=20, \rho-d\ge 14$. From achirality of these 10 lattices we obtain the required claim applying Lemma
\ref{A1-stabilization}.
\end{proof}

\subsection{A few more chiral lattices} We start from analysis of lattices in the first row of Table \ref{odd-lattices}.

\begin{proposition}\label{spherical-lattices}
Lattices  $-A_1+\la6\ra+t A_1$ are chiral for $t=0,1$ and achiral for $t\ge2$.
\end{proposition}

\begin{figure}
\caption{Vinberg's vectors and their Coxeter's graph for $-A_1+\la6\ra+A_1$}\label{9S}
\hbox{\vtop{\hsize38mm
\includegraphics[width=0.5\textwidth]{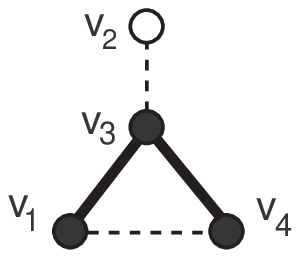}
}\hskip5mm
\vtop{\hsize4cm
$ \boxed{\begin{matrix}
&-A_1&\la6\ra&A_1\\
\hline
p&1&0&0\\
\text{level 0}&&&\\
v_1&0&1&0\\
v_2&0&0&1\\
\text{level 12}&&\\
v_3&3&-1&-3\\
v_4&3&-2&0\\
\end{matrix}}
$}}
\end{figure}

\begin{figure}
\caption{Vinberg's vectors and their Coxeter's subgraph for $-A_1+\la6\ra+2A_1$}
\label{8S}
\hbox{\vtop{\hsize4cm
\includegraphics[width=0.55\textwidth]{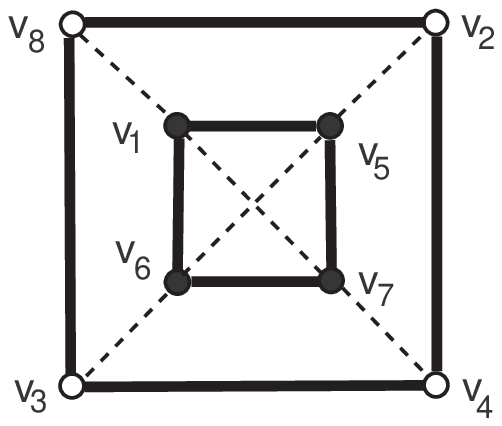}
}\hskip1cm
\vtop{\hsize5cm
$
\boxed{\begin{matrix}
&-A_1&\la6\ra&A_1&A_1\\
\hline
p&1&0&0&0\\
\text{level 0}&&&\\
v_1&0&1&0&0\\
v_2&0&0&1&0\\
v_3&0&0&0&1\\
\text{level 4}&&\\
v_4&1&0&-1&-1\\
\text{level 12}&&\\
v_5&3&-1&-3&0\\
v_6&3&-1&0&-3\\
v_7&3&-2&0&0\\
\text{level 16}&&\\
v_8&2&-1&-1&-1\\
\end{matrix}}
$}}
\end{figure}

\begin{proof}
For $\LL=-A_1+\la6\ra+A_1$ (the case $t=1$),
Vinberg's algorithm gives easily four root vectors which together with the corresponding
Coxeter's graph are shown on Figure \ref{9S}.
This graph describes a quadrilateral on the hyperbolic plane with consecutive sides defined by the roots
$v_1$, $v_3$, $v_4$, $v_2$;
two of its vertices (defined by the parabolic edges $v_1,v_3$ and $v_3,v_4$) lie at infinity
and the other two
vertices (defined by elliptic edges) are finite. The area of this quadrilateral is finite, which according to Vinberg's
criterion \ref{finiteness} means that Coxter's graph that we found is complete.
By Theorem \ref{existence of symmetry}
its only non-trivial $P$-direct automorphism
keeps fixed $v_2, v_3$ and interchanges
$v_1$ and $v_4$.
This automorphism is chiral, because it
does not change $v_3$ and due to Proposition \ref{reversing}
induces a trivial automorphism of $\discr_3(\LL)=\Z/3$.

Similar calculations for $\LL= -A_1+\la6\ra+2A_1$ (the case $t=2$)
give the result shown on Figure \ref{8S}.
The vertices of the graph obtained span $\LL$, and by this reason
the reflection
over the vertical mirror line of the graph induces a $P$-direct automorphism of $\LL$. This automorphism is $\Z/3$-reversing by Proposition \ref{reversing}, since
$v_1$ and $v_5$ are not congruent modulo $3\LL$.
 Applying Theorem \ref{existence of symmetry} we conclude that
$\LL$ is achiral (note that this application of the theorem does not require knowledge of the complete Coxeter's graph).

To deduce the result
 for
 other values of $t$, we use Lemma \ref{A1-stabilization}.
\end{proof}

Analysis of chirality of lattices with $\rho+d\le12$ requires case-by-case
study of the ones on the line  $\rho+d=12$.

\begin{figure}
\caption{Vinberg's vectors and their Coxeter's graph for $U+A_2+A_1+D_4$}\label{3-5}
\hbox{\vtop{\hsize 7cm
\raisebox{9mm}{\hskip8mm\includegraphics[width=0.57\textwidth]{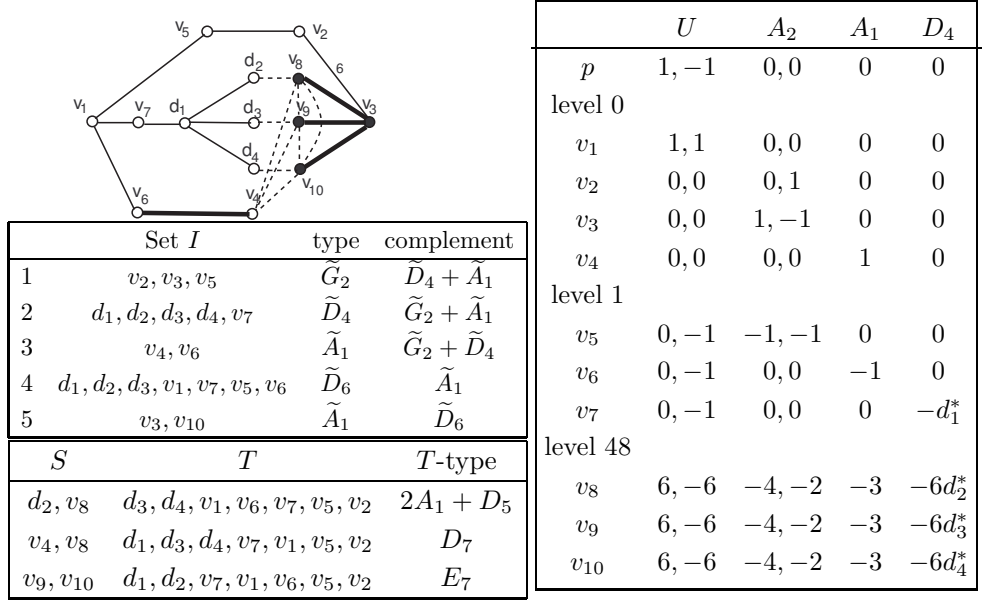}}
\vskip-11mm
\resizebox{70mm}{!}{
\begin{tabular}{|cccc|}
\hline
&Set $I$&type&complement\\
\hline
1&$v_2,v_3,v_5$&$\til G_2$&$\til D_4+\til A_1$\\
2&$d_1,d_2,d_3,d_4,v_7$&$\til D_4$&$\til G_2+\til A_1$\\
3&$v_4,v_6$&$\til A_1$&$\til G_2+\til D_4$\\
4&$d_1,d_2,d_3,v_1,v_7,v_5,v_6$&$\til D_6$&$\til A_1$\\
5&$v_3,v_{10}$&$\til A_1$&$\til D_6$\\
\hline
\end{tabular}}
\resizebox{70mm}{!}{
\begin{tabular}{|ccc|}
\hline
$S$&$T$&$T$-type\\
\hline
$d_2,v_{8}$&$d_3,d_4,v_1,v_6,v_7,v_5,v_2$&$2A_1+D_5$\\
$v_4,v_8$&$d_1,d_3,d_4,v_7,v_1,v_5,v_2$&$D_7$\\
$v_9,v_{10}$&$d_1,d_2,v_7,v_1,v_6,v_5,v_2$&$E_7$\\
\hline
\end{tabular}}
}
\vtop{\hsize4cm
$ \boxed{\begin{matrix}
&U&A_2&A_1&D_4\\
\hline
p&1,-1&0,0&0&0\\
\text{level 0}&&&&\\
v_1&1,1&0,0&0&0\\
v_2&0,0&0,1&0&0\\
v_3&0,0&1,-1&0&0\\
v_4&0,0&0,0&1&0\\
\text{level 1}&&\\
v_5&0,-1&-1,-1&0&0\\
v_6&0,-1&0,0&-1&0\\
v_7&0,-1&0,0&0&-d_1^*\\
\text{level 48}&&\\
v_8&6,-6&-4,-2&-3&-6d_2^*\\
v_9&6,-6&-4,-2&-3&-6d_3^*\\
v_{10}&6,-6&-4,-2&-3&-6d_4^*
\end{matrix}}$}}
\end{figure}

\begin{figure}
\caption{Vinberg's vectors and their Coxeter's graph for $-A_1+\la6\ra+E_8$}\label{4-5}
\hbox{\vtop{\hsize7cm
\raisebox{2mm}{\hskip12mm
\includegraphics[width=0.55\textwidth]{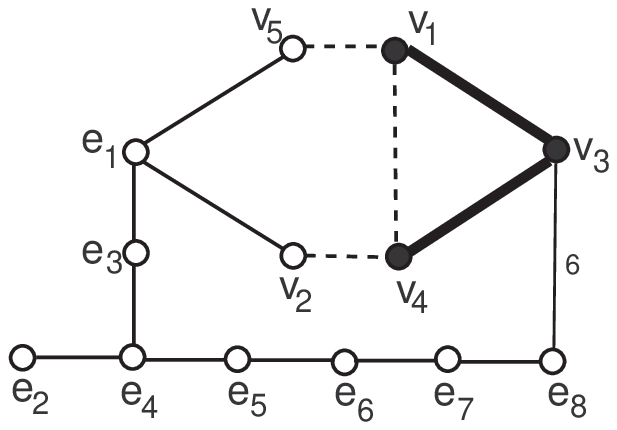}}
\vskip-3mm
\resizebox{70mm}{!}{
\begin{tabular}{|cccc|}
\hline
&Set $I$&type&complement\\
\hline
1&$v_2,e_1,e_2,e_3,e_4,e_5,e_6,e_7$&$\til E_7$&$\til A_1$\\
2&$v_1,v_3$&$\til A_1$&$\til E_7$\\
3&$v_2,v_5,e_1,e_2,e_3,e_4,e_5$&$\til D_6$&$\til G_2$\\
4&$v_3,v_7,e_8$&$\til G_2$&$\til D_6$\\
\hline
\end{tabular}}
\resizebox{70mm}{!}{
\begin{tabular}{|ccc|}
\hline
$S$&$T$&$T$-type\\
\hline
$v_1,v_{4}$&$e_1,e_2,e_3,e_4,e_5,e_6,e_7,e_8$&$E_8$\\
$v_1,v_5$&$v_2,e_2,e_3,e_4,e_5,e_6,e_7,e_8$&$A_1+D_7$\\
\hline
\end{tabular}}
}

\vtop{\hsize5cm
$ \boxed{\begin{matrix}
&-A_1&\la6\ra&E_8\\
\hline
p&1&0&0\\
\text{level 0}&&&\\
v_1&0&1&0\\
\text{level 4}&&&\\
v_2&1&0&-e_8^*\\
\text{level 12}&&\\
v_3&3&-1&-3e_8^*\\
v_4&3&-2&0\\
\text{level 16}&&\\
v_5&2&-1&-e_1^*\\
\end{matrix}}
$}}
\end{figure}

\begin{proposition}\label{odd-chiral}
Lattices
$U+A_2+D_4+A_1$
and  $-A_1+\la6\ra+E_8$ (from the fourth
and fifth rows of Table  \ref{odd-lattices})
are chiral.
\end{proposition}

\begin{proof}
The result of calculation of Vinberg's vectors and Coxeter's graphs for these two lattices are shown on
Fig. \ref{3-5} and \ref{4-5} respectively.
Each of these graphs
has an obvious symmetry:
\begin{itemize}
 \item
 $S_3$-symmetry on Fig. \ref{3-5} permuting 3 pairs of vertices: $d_i, v_{6+i}$, $i=2,3,4$;
\item
$S_2$-symmetry on Fig. \ref{4-5} permuting 2 pairs of vertices: $v_5, v_1$, and $v_2, v_4$.
\end{itemize}
We take into account this symmetry when we
apply Proposition \ref{sufficient-criterion} to check that our calculation of Vinberg's vectors is complete.
The details of this check (based on application of Theorem \ref{finiteness} and Proposition \ref{sufficient-criterion}) are shown on the
leftmost tables
of Figures \ref{3-5} and \ref{4-5} (right below Coxeter's graphs).
In the upper tables we
list: in the first column, the
sets of vertices
$I$
that generate a connected parabolic subgraph, $\G_{I}$,
one set in each symmetry class; in the second column, the types of these subgraphs;
and, in the third column, the types of the parabolic subgraphs that complement $\G_{I}$
to a rank $(n-1)$ parabolic graph.
In the bottom leftmost tables
we list the sets (denoted by $T$) of
vertices of elliptic complements to the dotted edges (whose endpoint pairs are denoted by $S$),
taking again just one representative from each symmetry class.

To conclude, we note
that although Coxeter's graphs on Fig. \ref{3-5} and \ref{4-5} have symmetries that interchange 6-roots, such symmetries
do not change the $A_2$-component (in Fig. \ref{3-5}), or the $\la6\ra$-component (in Fig. \ref{4-5}) modulo 3. Thus,
by Proposition \ref{reversing} these symmetries are $\Z/3$-direct and
the corresponding lattices are chiral.
\end{proof}

\begin{cor}\label{chiral-below}
In Table  \ref{odd-lattices}
all lattices in the fourth row and below are chiral provided  $\rho+d\le12$.
\end{cor}

\begin{proof}
Only the first seven rows of Table \ref{odd-lattices} contain some lattices with $\rho+d\le12$.
The claim for the fourth and fifth rows follows from Proposition \ref{odd-chiral} and  Lemma \ref{A1-stabilization}.
For the sixth and seventh row,  $\rho+d\le12$ holds only for $t=0$, that is for lattices
$-A_1+A_2+E_8$ and $U+A_2+E_8$ which were analyzed in \cite{chirality}, and shown
there
to be achiral.
\end{proof}

\subsection{Extension of the root system}\label{4-roots-added}
In this subsection, to prove chirality of lattices we apply a bit different method.
Namely, we enlarge the root system by adding 4-roots to 2- and 6-roots.
This extends the reflection group with new reflections (defined by 4-roots) and leads to subdivision of initial fundamental polyhedra
(defined by 2- and 6-roots) into more simple fundamental polyhedra corresponding to the extended reflection group.
The following lemma shows that we still keep control over the symmetries of the initial fundamental polyhedra.

\begin{lemma}\label{reflection_subgroup}
Let $W$ be a subgroup of a discrete group $G$ generated by reflections in a hyperbolic space $\Lambda$.
Then, for any fundamental polyhedron $\Pi$ of $G$, there is only one fundamental polyhedron $P$ of $W$ containing $\Pi$,
and any $f\in \Aut(P)$ can be written as a product $gh_1\dots h_n$ where $g\in \Aut(\Pi)$ and each of $h_i$ is a reflection in a facet of $\Pi$
which is not a facet of $P$.
\end{lemma}
\begin{proof} Straightforward from
$\hat W =W \rtimes Aut(\Pi)$ where $\Aut(\Pi)$ is the group of symmetries of $\Pi$ and $\hat W$ is the normalizer of $W$ in the isometry group of $\Lambda$.
\end{proof}

To find a sequence of roots defining a fundamental polyhedron of the extended reflection group we use Vinberg's algorithm like before. In corresponding Coxeter's graphs
the 4-roots are marked by encircled black squares.

\begin{proposition}\label{4-roots-method}
Lattices $-A_1+A_2+4A_1$ and $-A_1+A_2+A_1+D_4=U+A_2+4A_1$ are chiral.
\end{proposition}

\begin{figure}\label{2-5-4roots}
\caption{Vinberg's vectors and their Coxeter's graph for $-A_1+A_2+A_1+D_4=U+A_2+4A_1$
}\label{2-5-4roots}
\hbox{\vtop{\hsize6.5cm
\raisebox{15mm}{\hskip5mm
\includegraphics[width=0.5\textwidth]{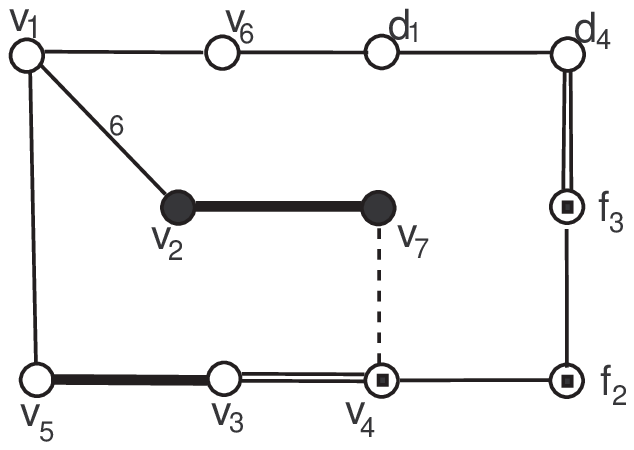}}
\vskip-15mm
\resizebox{64mm}{!}{
\begin{tabular}{|cccc|}
\hline
&Set $I$&type&complement\\
\hline
1&$v_5,v_1,v_2$&$\til G_2$&$\til F_4$\\
2&$v_6,v_1,v_2$&$\til G_2$&$\til C_4$\\
3&$v_6,d_1,d_4,f_3,f_2$&$\til F_4$&$2\til A_1$\\
4&$d_1,d_4,f_3,f_2,v_4$&$\til F_4$&$\til G_2$\\
5&$d_4,f_3,f_2,v_4,v_3$&$\til C_4$&$\til G_2$\\
\hline
\end{tabular}}
\resizebox{64mm}{!}{
\begin{tabular}{|ccc|}
\hline
$S$&$T$&$T$-type\\
\hline
$v_4,v_{7}$&$v_1,v_5,v_6,d_1,d_4,f_3$&$B_6$\\
\hline
\end{tabular}
}}
\vtop{\hsize5cm
$ \boxed{\begin{matrix}
&-A_1&A_2&A_1&D_4\\
\hline
p&1&0,0&0&0\\
\text{level 0}&&\\
v_1&0&0,1&0&0\\
v_2&0&1,-1&0&0\\
v_3&0&0,0&1&0\\
d_1&0&0,0&0&d_1\\
f_2&0&0,0&0&d_2-d_3\\
f_3&0&0,0&0&d_3-d_4\\
d_4&0&0,0&0&d_4\\
\text{level 2}&&\\
v_4&1&0,0&-1&-2d_2^*\\
\text{level 4}&&\\
v_5&1&-1,-1&-1&0\\
v_6&1&-1,-1&0&-d_1^*\\
\text{level 12}&&\\
v_7&3&-4,-2&0&0\\
\end{matrix}}
$}}
\end{figure}

\begin{proof}
We extend the group $W$
generated by the reflections $R_v$, $x\mapsto
x-2\frac{ex}{e^2}v$,  in the $2$- and $6$-roots,
up to a group $G$ generated, in addition, by the
reflections $x\mapsto x-2\frac{ex}{e^2}e$ in the $4$-roots $e$, that is $e\in \L$
with $e^2=4, e\cdot \L=0\mod 2$.
Applying Vinberg's algorithm to the group $G$
we obtain the lists of roots shown
in the rightmost tables of Figures \ref{2-5-4roots} and \ref{1-5},
and derive from them the corresponding Coxeter's graphs.
Completeness of these root-lists follows easily from
Theorem \ref{finiteness} and
Proposition \ref{sufficient-criterion}: see the list of connected parabolic subgraphs supplied with the type of
parabolic subgraphs that complement them to a rank $n-1$ parabolic graph, and the list of the elliptic complements to Lann\'er's subgraphs (here, they
are limited to dotted edges)

\begin{figure}[h!]
\caption{Vinberg's vectors and their Coxeter's graph for $-A_1+A_2+4A_1$
}\label{1-5}
\hbox{
\vtop{\hsize6cm
\raisebox{15mm}{\includegraphics[width=0.5\textwidth]{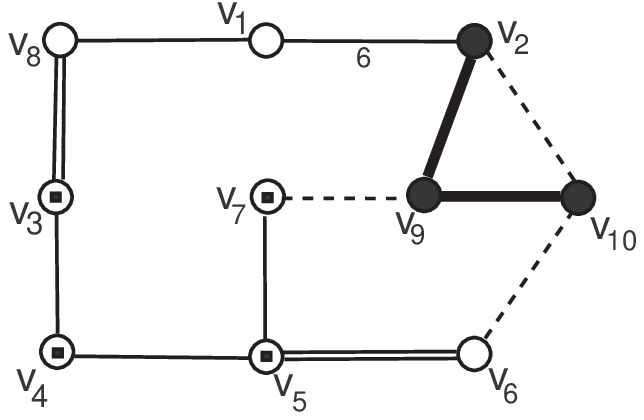}}
\vskip-15mm
\resizebox{54mm}{!}{
\begin{tabular}{|cccc|}
\hline
&Set $I$&type&complement\\
\hline
1&$v_8,v_1,v_2$&$\til G_2$&$\til B_3$\\
2&$v_4,v_5,v_6,v_7$&$\til B_3$&$\til G_2$\\
3&$v_2,v_9$&$\til A_1$&$\til C_4$\\
4&$v_8,v_3,v_4,v_5,v_6$&$\til C_4$&$\til A_1$\\
5&$v_9,v_{10}$&$\til A_1$&$\til F_4$\\
6&$v_1,v_8,v_3,v_4,v_5$&$\til F_4$&$\til A_1$\\
\hline
\end{tabular}
}
\resizebox{62mm}{!}{
\begin{tabular}{|ccc|}
\hline
$S$&$T$&$T$-type\\
\hline
$v_2,v_{10}$&$v_3,v_4,v_5,v_7,v_8$&$B_5$\\
$v_7,v_9$&$v_1,v_3,v_4,v_8,v_6$&$F_4+A_1$\\
$v_6,v_{10}$&$v_1,v_8,v_3,v_4,v_7$&$F_4+A_1$\\
\hline
\end{tabular}
}}
\vtop{\hsize5cm
\resizebox{64mm}{!}{
$ \boxed{\begin{matrix}
&-A_1&A_2&A_1&A_1&A_1&A_1\\
\hline
p&1&0,0&0&0&0&0\\
\text{level 0}&&\\
v_1&0&0,1&0&0&0&0\\
v_2&0&1,-1&0&0&0&0\\
v_3&0&0,0&1&-1&0&0\\
v_4&0&0,0&0&1&-1&0\\
v_5&0&0,0&0&0&1&-1\\
v_6&0&0,0&0&0&0&1\\
\text{level 2}&&\\
v_7&1&0,0&-1&-1&-1&0\\
\text{level 4}&&\\
v_8&1&-1,-1&-1&0&0&0\\
\text{level 12}&&\\
v_9&3&-4,-2&0&0&0&0\\
\text{level 108}&&\\
v_{10}&9&-8,-4&-3&-3&-3&-3\\
\end{matrix}}$}}}
 \end{figure}

None of these graphs has a non-trivial symmetry. Hence, according to Lemma \ref{reflection_subgroup},
the reflections defined by the 4-roots
from our lists preserve a certain fundamental polyhedron $P$ of $W$ and
generate all the symmetries of $P$.
The action of each of these 4-reflections on the $A_2$-component is trivial modulo 3. Due to Proposition \ref{reversing} this implies that
for both lattices all the $P$-direct automorphisms
are $\Z/3$-preserving.
\end{proof}

\begin{cor}\label{2-3rows}
Lattices in the second and third rows of table \ref{odd-lattices} with $\rho+d\le12$ are chiral.
\end{cor}

\begin{proof}
It follows from Proposition \ref{4-roots-method} and Lemma \ref{A1-stabilization}.
\end{proof}

\subsection{Proof of Theorem \ref{achirality-odd-complete}}\label{proof-odd-complete}
Lemma \ref{A1-stabilization} shows existence of a border between the chiral and achiral lattices in each row of Table \ref{odd-lattices}.
For the first row the border
is found in Proposition \ref{spherical-lattices}.
For the other rows, the border is the line $\rho+d=14$, since for $\rho+d\ge14, \rho-d\ge 2$ the lattices are achiral
by Proposition \ref{achirality-odd}, and for $\rho+d<14$ are chiral by Corollaries \ref{chiral-below} and \ref{2-3rows}. \qed

\subsection{Proof of Theorem \ref{main}}
Theorem \ref{chirality-theorem} reduces the question on chirality of non-singular real cubic fourfolds $X$ to analysis of their lattices $\M_+^0(X)$.
These lattices are listed in Tables \ref{even-lattices} (even lattices) and \ref{odd-lattices} (odd lattices). The required analysis is performed in
Theorems \ref{main-even} and \ref{achirality-odd-complete}
and yields 18 chiral lattices described in Theorem \ref{main}. \qed


\section{Concluding Remarks}
\subsection{Topological chirality}\label{weak}
It is natural to ask if it is possible for some chiral real non-singular cubic
fourfold $X\subset P^5$ to distinguish it from its mirror image
just by the topology of embedding of $X_\R$  into $P^5_\R$. More precisely, is it possible that $X_\R$ and its mirror image are not isotopic?
The answer turns out to be in the
negative.
Indeed, all the coarse deformation classes shown on the upper-right side of Table \ref{chirality-all} are achiral, and hence,
for any $X$ from these classes, $X_\R$ is isotopic to its mirror image. On the other hand,
according to
\cite[Corollary 3.3.3]{topology},
starting from these classes and performing surgeries through cuspidal-strata only, one can reach all the other classes
except, in notation of \cite{topology}, $\Cal C^{10,1}$ and $\Cal C^{2,1}_I$ (that is, in Table \ref{chirality-all}, the classes
with even parity and $(\rho,d)$ equal to $(20,0)$ and $(12,8)$), which are achiral by
the results of the present paper. Hence, there remain to notice that a surgery through a cuspidal-stratum
is nothing but replacement of a small ball in $X_\R\subset P^5_\R$ by its ambient connected sum with $S^1\times S^3$ embedded in
a standard "unknotted" way (see \cite[Lemma 3.2.1]{topology}).
Since $X_\R$ is connected and its embedding into $P^5_\R$ is one-sided, it does not matter
to which side of $X_\R$ we attach $S^1\times S^3$.

\subsection{Pointwise achirality}\label{strong}
If a real cubic $X\subset P^5$
is symmetric with respect to some hyperplane in $P^5_\R$, then $X$ is obviously achiral.
A naturally arising question is whether the converse is true:
{\it if a coarse deformation class is achiral, does it
contain a representative which is symmetric with respect to some
mirror reflection}?
The methods of this paper can be developed further to respond to this question as well, but
requires an essential additional work.
Already first inspection shows existence
of such a representative
(symmetric across a hyperplane) in each of
the 3 achiral classes from the two bottom lines of Table 1 (that is all the achiral classes of
$M$- and  $(M-1)$-cubics).

Another approach is to look for explicit mirror-symmetric equations. For example, such an equation can be always written
in an appropriate coordinate system in the form $f_3(x_0,\dots, x_4) + x_5^2 f_1(x_0,\dots, x_4)=0$
where $x_5=0$ is the mirror-hyperplane and $f_3=0$ defines a real non-singular threefold cubic in $P^4$ transversal to a real hyperplane $f_1=0$
(compare, f.e., \cite{LPZ}).
This led us to a recurrent procedure that produces a sequence of maximal symmetric real non-singular cubic hypersurfaces
in consecutive dimensions (the details are to be exposed in a separate publication).

Note finally
that as was already pointed in \cite{chirality},
the equation $t(x_0^3+\dots +x_5^3) - (x_0+\dots + x_5)^3=0$
provides symmetric (across hyperplanes $x_i=x_j$) representatives for
4 achiral classes of cubics. Each of these classes is determined by topology of the real locus which is as follows:
$\Rp4$ for $t<0$ and $t>36$,
$\Rp4\+S^4$ for $16<t<36$, $\Rp4\#5(S^1\times S^3)$ for $4<t<16$, and $\Rp4\#{10}(S^2\times S^3)$ for $0<t<4$.

\end{document}